\documentclass[12pt]{article}
\usepackage[margin=1in]{geometry}
\usepackage{graphicx}
\usepackage{subcaption}
\graphicspath{{./plots/}}
\usepackage{float}
\usepackage{amsmath}
\usepackage[none]{hyphenat}
\usepackage[backref=page,hidelinks]{hyperref}
\renewcommand*{\backref}[1]{}
\renewcommand*{\backrefalt}[4]{%
	\ifcase #1 (Not cited.)%
	\or        (Cited on page~#2.)%
	\else      (Cited on pages~#2.)%
	\fi}
\usepackage[numbers]{natbib}
\usepackage{amssymb}
\usepackage{url} 

\usepackage{amsthm}
\newtheorem{prop}{Proposition}
\newtheorem{Def}{Definition}
\newtheorem{Thm}{Theorem}
\newtheorem*{Thm*}{Theorem}
\newtheorem{Cor}{Corollary}

\newtheorem*{Egs*}{Example}

\usepackage[nottoc,notlot,notlof]{tocbibind}
\usepackage[title,titletoc]{appendix}

\usepackage{makecell}

\usepackage{enumitem}

\usepackage[none]{hyphenat}

\newcommand{\diff}{\mathrm{d}}
\newcommand{\N}{\mathbb{N}}
\newcommand{\R}{\mathbb{R}}
\newcommand{\E}{\mathbb{E}}
\renewcommand{\P}{\mathbb{P}}
\renewcommand{\bf}{\mathbf}

\makeatletter
\g@addto@macro\bfseries{\boldmath}
\makeatother

\makeatletter
\newcommand*{\rom}[1]{\expandafter\@slowromancap\romannumeral #1@}
\makeatother

\usepackage{enumitem} 

\begin{document}
	
	\title{A Study of the Probability Distribution of the Balls in Bins Process with Power Law Feedback}
	\author{Samuel Forbes\thanks{forbessam2@gmail.com} 
		}
	\date{August 2023}
	
	\maketitle
	
	\begin{abstract}
		We analyse the balls in bins process with feedback with primary focus on the power law feedback function $f(\omega)=\eta \omega^{\gamma}\,$, $\eta>0\,$ $\gamma \geq0\,$. Using the recursive solution to the master equation we find for power law feedback numerical evidence that the probability mass function for finite time scales asymptotically with  $\omega^{-\gamma}\,$ for $\gamma>1\,$. We also provide simulations supporting a previous result by Oliveira (corollary to Theorem 4 in \cite{oliveira2009onset}) that the tail of the losers scale as $\omega^{-(\gamma-1)}\,$  but extending to $N \geq 2$ bins. We thus find evidence that the balls in bins process with power law feedback produces power law distributions as is common to many real world phenomena.
	\end{abstract}
	
	\renewcommand{\abstractname}{Acknowledgements}
	\begin{abstract}
		Samuel Forbes would like to thank Stefan Grosskinsky for valuable assistance with the background probability theory, simulations and comments on the manuscript, Paul Chleboun for important assistance with the binary tree search implementation to produce Figure \ref{urn_cpp_tails} and financial support from EPSRC through grant EP/L015374/1 when part of this work was studied.
	\end{abstract}
	
	\section{Introduction}
	
	Suppose there are $N \in \N$ bins and that the number of balls in bin $j \in \{1,\,2, \dots,N\}$ at iteration $n=0,\,1,\,2,\dots$ is $I_n(j)\in \mathbb{N}\,$. We shall refer to the bins as \textbf{agents} as this model may be applied to many types of real world phenomena, examples of which will be discussed below.
	Define the set of all agents at iteration $n$ as $\bf{I}_n := \{I_n(1), I_n(2), \dots, I_n(N)\}$. Then suppose the probability that agent $j$ gains one ball at iteration $n$ is
	\begin{equation}
	\P(I_{n+1}(j) = I_{n}(j)+ 1|\bf{I}_{n}) = \frac{f(I_{n}(j))}{
		\sum_{i=1}^{N}f(I_{n}(i))}\,
	\label{update_rule}
	\end{equation}
	where $f:\N \rightarrow \R_{>0}$ is a continuous positive \textbf{feedback function}.
	The \textbf{balls in bins process with feedback} \cite{drinea2002balls, oliveira2008balls}, also called the generalised P\'{o}lya urn with feedback \cite{pemantle2007survey, gottfried2023asymptotics} is the Markov chain $\{\bf{I}_n \,: \, n \in \N\}$ which we shall refer to as the \textbf{feedback model}. We shall focus on the following \textbf{power law feedback function}
	\begin{equation}
	f(\omega) = \eta \omega^{\gamma} \quad \text{ such that } \eta > 0\,, \gamma \geq0\,.
	\label{pl_feedback_fn}
	\end{equation}
	
	A feature of the feedback model is the presence of \textbf{monopoly} where one agent after a certain time gains every new ball, see Theorem 1.1 in \cite{oliveira2008balls}. Monopoly is something seen in the real world, for example the extreme wealth of the owner(s) of a company with market domination such as Google. Another example is the city of  London which has a comparatively vast population compared to other cities in the UK. However the monopoly in real life variables is usually not as strong as in the feedback model and in reality there is often more than one extremely sizeable agent.
	 
	 Another feature of the feedback model that occurs alongside monopoly is a \textbf{power law distribution} in the tail of the loser (non-monopoly) agents. In particular for the power law feedback function \eqref{pl_feedback_fn} and $N=2$ agents the tail of a loser agent scales proportional to $\omega^{-(\gamma-1)}$ which is a corollary to Theorem 4 in \cite{oliveira2009onset}. We provide simulations in Section \ref{sec_tail_sim} to support that this is also true for $N \geq 2$ agents. Many real world phenomena have power law distributions and this may be due to a feedback mechanism and thus the feedback model could be applicable to these cases \cite{drinea2002balls, pemantle2007survey, gottfried2023asymptotics, oliveira2009onset}.
	  In fact the feedback model has been applied directly to US wealth data from 1988-2012 with power law feedback \eqref{pl_feedback_fn} \cite{vallejos2018agent}. Both wealth and city size exhibit power laws in the probability distribution, see for example \cite{vermeulen2018fat,vallejos2018agent, forbes2022study} and \cite{mori2020common} respectively. 
	 Both power laws and monopoly are seen in the empirical tail, defined as one minus the empirical cumulative distribution function, of simulations\footnote{All simulations and numerical solutions in this paper can be found in \cite{sam_git_code}. } of the feedback model, see Figure \ref{urn_cpp_tails}.
	
	\begin{figure}[H]
		\centering
		\captionsetup{justification=centering}
		\includegraphics[width=0.7\textwidth]{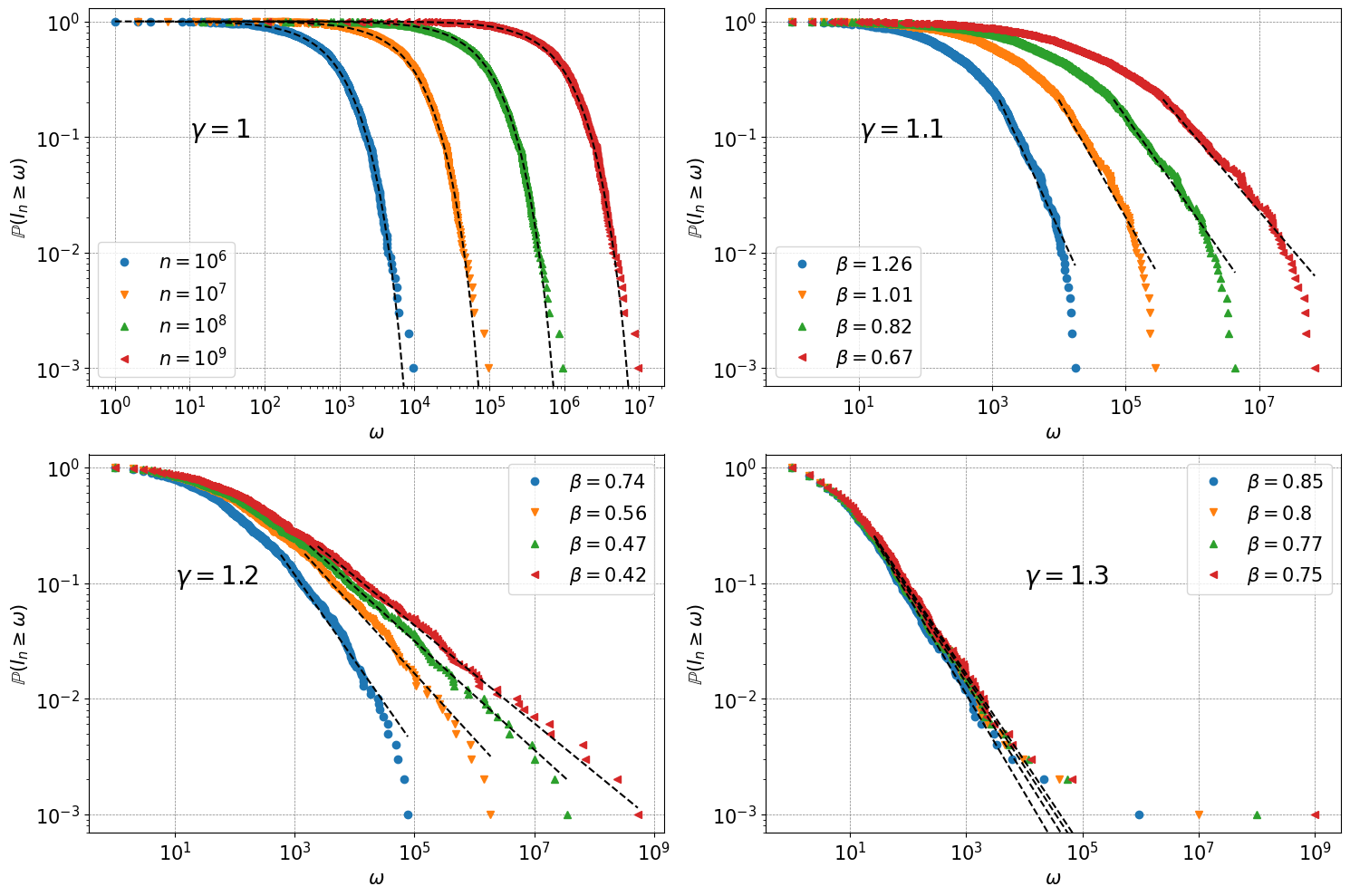}
		\caption{Empirical tail of a single simulation of feedback model with feedback function $f(\omega)=\omega^{\gamma}$ for $\gamma=1,\,1.1,\,1.2,\,1.3$, (top left, top right, bottom left, bottom right respectively), $N=1000$ agents, $\bf{I}_0=\bf{1}_{N}=\{1,1,\dots,1\} $ and at iteration $n=10^6,\,10^7,\,10^8,\,10^9\,$. Black lines in top left are exponential fits. Black lines for $\gamma>1$ are power laws proportional to $\omega^{-\beta}$ fitted with the MLE method found in \cite{clauset2009power} approximating the discrete values as continuous. Monopoly agents are most obvious for bottom right with $\gamma=1.3\,.$ }
		\label{urn_cpp_tails}
	\end{figure}
	
	We shall now introduce an equivalent form of the feedback model in continuous time that will allow us to find the probability mass function solution. Suppose the set of $\bf{W}_t=\{W_t(1), W_t(2), \dots, W_t(N)\}$ is a continuous time homogeneous Markov chain where agent $j$ has $W_t(j) \in \N$ balls at time $t \in [0,\infty)\,.$ Suppose also 
	$\bf{W}_0=\bf{I}_0\,.$
	For convenience let us consider a single agent and define $W_t:=W_t(j)\,$. 
	Define $p_t(\omega,\nu) :=\P(W_t=\nu|W_0=\omega)$ as the transition mass function of $W_t$ and $p_t(\omega):=\P(W_t=\omega)$ as the mass function of $W_t\,.$ Define the \textbf{rates} as $g(\omega,\nu) := \frac{\diff}{\diff t} p_t(\omega,\nu) \big|_{t=0}\,$. We set the rates as 
	\begin{equation}
	g(\omega,\nu) = \begin{cases}
	f(\omega) & \nu=\omega+1 \,, \\
	-f(\omega) & \nu=\omega \,, \\
	 0 & \text{otherwise} 
	\end{cases}
	\label{rates}
	\end{equation}
	where $f$ is the feedback function as above.

	Define the holding time $H_{\omega}$ as the time until $W_t$ jumps up by one ball. Thus
	for some time $\tau \in [0,\infty)$
	\begin{equation*}
	H_{\omega} := \inf_{t} \{W_{\tau+t} = \omega +1 : W_{\tau}= \omega \}\,.
	\end{equation*}
	We note, see e.g. Section 6.9 of \cite{stirzaker1992probability}, the holding time $H_{\omega}$ follows an exponential distribution with mean 
	$1/|g(\omega,\omega)|=1/f(\omega)$ and thus $H_{\omega} \sim \text{Exp}(f(\omega))\,.$ We now define the \textbf{jump times} $t_0=0<t_1<t_2<\dots $ as the cumulative sum of the holding times which occur when the Markov chain $\bf{W}_t$ changes to a new state where one of the agents gains one ball. We note the agents $W_t$ within $\bf{W}_t$ \textbf{jump independently} of each other. Then by \textbf{exponential embedding}, see Theorem 2 in \cite{oliveira2009onset}, it can be shown that we have $\bf{W}_{t_n}=\bf{I}_n\,$ so that $\bf{W}_t$ is equal to the feedback model $\bf{I}_n$ at the jump times $t_n$. We shall refer to the \textbf{pure birth process}, see Chapter \rom{17}, Section 3 of \cite{feller1967introduction},
	$\{\bf{W}_t\,:\, t=t_n\}$ as the \textbf{feedback process}.
	
	For $W_0=\omega_0$ the \textbf{master equation}, see e.g. Section 3 of \cite{gardiner1985handbook} and   Chapter \rom{17}, Section 3 of \cite{feller1967introduction} specifically for pure birth processes, for the time derivative of the probability mass function for an agent in $\bf{W}_t$ with the rates \eqref{rates} is	
	\begin{equation}
	\frac{\diff}{\diff t}p_t(\omega)= \begin{cases}
	-p_t(\omega)f(\omega) & \omega = \omega_0  \\
	p_t(\omega-1)f(\omega-1)-p_t(\omega)f(\omega) & \omega> \omega_0 \,. \\
	\end{cases}
		\label{master_eqn}
	\end{equation}

	The master equation has a recursive solution found in \cite{severo1969recursion}. 
	We show in Section \ref{sec_master_eqn} that the first term of this solution when applied to power law feedback \eqref{pl_feedback_fn} with $\gamma>1$ scales asymptotically as $\omega^{-\gamma}\,$ and is a good approximation under certain parametrisations. We emphasise that our new work is found in Sections \ref{sec_master_eqn} and \ref{sec_tail_sim}.

	\section{Overview of Some Existing Results}
	
	\label{sec_overview}
	
	As mentioned in the introduction, interesting properties of the feedback model are already known. In particular, conditions on the feedback function $f(\omega)$ leads to monopoly where one agent gains all but a finite number of balls and fixed rankings where the order of agents from richest to poorest (most balls to least) becomes fixed after a certain period of time. As these cases may have broad applications to models of real world phenomena as discussed in the introduction we reproduce them in this section.
	
	\begin{Def}[Monopoly]
	Monopoly occurs if there exists an agent $j_m \in \{1,2,\dots,N\}$ and $n_m \in \mathbb{N}$ such that for all $n>n_m$ we have $\P(I_{n+1}(j_m)=I_n(j_m)+1|\mathbf{I}_n)=1\,.$
	\label{def_monopoly}
	\end{Def}
	
	We shall refer to the agent $j_m$ as a \textbf{monopoly agent}.
	
	\begin{Def}[Losers]
		We define the random variable $\textbf{L}$ as the number of balls of the $N-1$ `loser' non-monopoly agents once monopoly has occurred. 
		\label{def_losers}
	\end{Def}
	
	To be precise for  $n_m \in \N$ as in Definition \ref{def_monopoly} then all
	$j_l \in \{1,2,\dots,N\}$ such that $j_l \neq j_m $ are the loser agents. Then for all $n>n_m$ we have $\P(I_{n+1}(j_l)=I_n(j_l)+1|\mathbf{I}_n)=0$ and thus 
	$I_n(j_l)=I_{n+1}(j_l)=I_{j_l} < \infty$ and so $\textbf{L}= \{I_{j_l}\}\,.$
	
	\begin{Def}[Fixed Rankings]
		Rankings of agents become fixed if there exists an $n_f \in \mathbb{N}$ such that for a particular ordering of the $N$ agents $j_i\,,i=1,2,\dots,N\,$
		\begin{equation*}
		I_n(j_1)<I_n(j_2) < \dots < I_n(j_N) \quad \forall \, n>n_f\,.
		\end{equation*}
	\end{Def}
	
	Clearly if monopoly occurs then so will fixed rankings.
	
	\begin{Thm}[Theorem 1.1 \cite{oliveira2008balls}]
		In the feedback model  we have the following three mutually exclusive cases irrespective of initial conditions:
		\begin{enumerate}
			\item If $\sum_{\omega=1}^{\infty}f(\omega)^{-1} <\infty $ then monopoly occurs.
			\label{feed_thm1}
			\item If $\sum_{\omega=1}^{\infty}f(\omega)^{-1} =\infty $ but $\sum_{\omega=1}^{\infty}f(\omega)^{-2} <\infty $ then monopoly does not occur but there will be fixed rankings.
			\label{feed_thm2}
			\item If $\sum_{\omega=1}^{\infty}f(\omega)^{-2} = \infty $ there will be no monopoly and no fixed rankings.
			\label{feed_thm3}
		\end{enumerate}
		\label{feedback_cases}
	\end{Thm}

	Theorem \ref{feedback_cases} with $f(\omega) = \eta \omega^{\gamma}$ \eqref{pl_feedback_fn} means that monopoly occurs with $\gamma>1$, no monopoly but fixed rankings when $1/2<\gamma \leq 1$ and no monopoly and no fixed rankings for $0 \leq \gamma \leq 1/2\,.$
	Examples of feedback functions satisfying the different cases in Theorem \ref{feedback_cases} are given in \cite{oliveira2008balls, gottfried2023asymptotics}.
	For a particular class of feedback functions, called valid feedback functions, satisfying Case \ref{feed_thm1}. of Theorem \ref{feedback_cases}, a Theorem on the tail of the losers has been found, see Theorem 4 in \cite{oliveira2009onset}. 
	In particular the following corollary, as mentioned in the abstract and introduction, applies this theorem to the power law feedback function \eqref{pl_feedback_fn}.
	\begin{Cor}[Corollary to Theorem 4 in \cite{oliveira2009onset}]
		Suppose there are $N=2$ agents and $f$ is the power law feedback function \eqref{pl_feedback_fn} with $\gamma>1$ then there exists an $\alpha>0$ such that\footnote{Where for general functions $g$ and $h\,$, we define $g(x) \simeq h(x) \Leftrightarrow g(x)/h(x) \rightarrow 1$ as $x \rightarrow \infty\,.$} $\mathbb{P}(L>\omega) \simeq \alpha /\omega^{\gamma-1}\,.$
		\label{cor_pl_two_ball}
	\end{Cor}
	We shall give evidence in Section \ref{sec_tail_sim} through simulations of the feedback process that Corollary \ref{cor_pl_two_ball} holds also for $N \geq 2$ agents.

    We shall now focus on probability mass function solutions to the master equation \eqref{master_eqn}.
    We write down two particular solutions applied to power law feedback \eqref{pl_feedback_fn} for $\gamma=0\,,\,1\,.$ 
	
	\begin{prop}[See Ch.\rom{17}, Sections 1-3 of \cite{feller1967introduction}]
		The solution to the master equation \eqref{master_eqn} with $W_0=\omega_0$ for power law feedback \eqref{pl_feedback_fn} for
		\begin{enumerate}
			\item $\gamma=0$ is Poisson with probability mass function
			\begin{equation*}
			p_t(\omega) = \frac{(\eta t)^{\omega-\omega_0}}{(\omega-\omega_0)!} e^{-\eta t}\,, \quad \omega=\omega_0, \omega_0+1, \dots 
			\end{equation*}
			Note this is the Poisson process.
			\label{master_case1}
			\item $\gamma=1$ is negative binomial with probability mass function 
			\begin{equation*}
			p_t(\omega) = {\omega-1 \choose \omega-\omega_0}
			(1-e^{-\eta t})^{\omega-\omega_0}e^{- \eta \omega_0 t}\,, \quad 
			\omega=\omega_0, \omega_0+1, \dots
			\end{equation*}
			Note this is the Yule process and for the particular case with $\omega_0=1$ 
			gives the geometric distribution which can be approximated by the exponential distribution, see Figure \ref{urn_sum_preds}.
			\label{master_case2}
		\end{enumerate}
		\label{Soln_partic_master}
	\end{prop}
	
	\begin{proof}
		Substitute the probability mass functions into the respective master equation \eqref{master_eqn} for power law feedback \eqref{pl_feedback_fn} for $\gamma = 0$ and $\gamma = 1\,$ . 
	\end{proof}
	The following corollary gives a general solution to the master equation \eqref{master_eqn} under a mild condition on the feedback function  $f\,.$ We provide the proof for this particular case as it is significantly simpler than to a general case which applies to a broader class of equations, see Theorem 1 in \cite{severo1969recursion}.
	\begin{Cor}[Corollary to Theorem 1 in \cite{severo1969recursion}]
		For $W_0=\omega_0$ and provided $f(\omega) \neq f(i)$ for all $i=\omega_0,\, \omega_0+1,\dots, \,\omega-1\,,$ the master equation \eqref{master_eqn} has the following solution
		\begin{equation}
		p_t(\omega) = \sum_{i=\omega_0}^{\omega} a_{\omega,i}e^{- f(i) t}\,, \quad \omega \geq \omega_0
		\label{master_sum_soln}
		\end{equation}
		with $a_{\omega_0,\omega_0} = 1$ and for $\omega >\omega_0$
		\begin{align*}
		a_{\omega,i} &= \frac{f(\omega-1)}{f(\omega)-f(i)}a_{\omega-1,i} \,, \quad i=\omega_0, \, \omega_0+1,\dots, \,\omega-1 \,, \\
		a_{\omega,\omega} &= - \sum_{i=\omega_0}^{\omega-1} a_{\omega,i} \,.
		\end{align*}
		\label{gen_mast_sum}
	\end{Cor}
	
	\begin{proof}
		The solution for $\omega=\omega_0$ to \eqref{master_eqn} is
		\begin{equation*}
		p_t(\omega_0)= e^{-f(\omega_0) t}
		\end{equation*}
		which agrees with the corollary. Now for $\omega>\omega_0$ we have
		substituting \eqref{master_sum_soln} into the LHS of the master equation \eqref{master_eqn} that
		\begin{equation*}
		\frac{\diff}{\diff t} p_t(\omega) = - \sum_{i=\omega_0}^{\omega} a_{\omega,i}f(i)e^{-f(i) t} 
		\end{equation*}
		and into the RHS we have with $a_{\omega-1,\omega}:=0$
		\begin{equation*}
		p_t(\omega-1)f(\omega-1)-p_t(\omega)f(\omega)=  \sum_{i=\omega_0}^{\omega}(f(\omega-1)a_{\omega-1,i}-f(\omega)a_{\gamma,\omega,i})e^{-f(i) t} \,.
		\end{equation*}
		Setting the summands of the LHS and RHS equal to each other we have for $\omega>\omega_0\,$
		\begin{align*}
		-f(i) a_{\omega,i} &= f(\omega-1)a_{\omega-1,i}-f(\omega)a_{\omega,i} \Leftrightarrow \\
		a_{\omega,i} &= \frac{f(\omega-1)}{f(\omega)-f(i)} a_{\omega-1,i}\,, \quad i= \omega_0, \, \omega_0+1, \dots, \omega-1 \,. 
		\end{align*}
		provided $f(\omega) \neq f(i)$ for all $i=\omega_0,\, \omega_0+1,\dots, \,\omega-1\,.$
		
		As $W_0=\omega_0$ we have that $p_0(\omega) = \sum_{i=\omega_0}^{\omega} a_{\omega,i} = 0$ for all $\omega>\omega_0\,.$ This implies $a_{\omega,\omega} = - \sum_{i=\omega_0}^{\omega-1} a_{\gamma,\omega,i}$
		for all $\omega>\omega_0\,.$ Finally we also have $p_0(\omega_0) =a_{\omega_0,\omega_0} = 1\,.$
		
	\end{proof}
	
	From \eqref{master_sum_soln} the probability any agent is finite in infinite time is zero which agrees with the fact that the feedback process is a pure birth process.
	\begin{Cor}
		For all
		$\omega \in \N$ 
		\begin{equation}
		p_{\infty}(\omega):=\lim_{t \rightarrow \infty} p_{t}(\omega) = 0 \,.
		\label{limit_mass_fn}
		\end{equation}
		\label{cor_mast_part}
	\end{Cor}
	
	\begin{Egs*}
		To illustrate Corollary \ref{gen_mast_sum} let us check the sum formula \eqref{master_sum_soln} for the power law feedback \eqref{pl_feedback_fn} with $\gamma=1$ ($f(\omega)=\eta \omega$) has the negative binomial probability solution found in Proposition \ref{Soln_partic_master}.  The negative binomial mass function in Proposition \ref{Soln_partic_master} can be expanded so that
		\begin{align*}
		p_t(\omega) &= {\omega-1 \choose \omega-\omega_0}
		(1-e^{-\eta t})^{\omega-\omega_0}e^{-\omega_0 \eta t} \\
		&= \sum_{i=\omega_0}^{\omega} {\omega-1 \choose \omega-\omega_0}{\omega -\omega_0 \choose i-\omega_0 }(-1)^{i-\omega_0}e^{-\eta i  t } \,.
		\end{align*}
		Taking 
		\begin{equation*}
		a_{\omega,i}={\omega-1 \choose \omega-\omega_0}{\omega -\omega_0 \choose i-\omega_0 }(-1)^{i-\omega_0} \quad  \text{for } i=\omega_0,\omega_0+1,\dots,\omega
		\label{geom_coef}
		\end{equation*}
		we shall show that these coefficients follow the rules in Corollary \ref{gen_mast_sum}. First we see $a_{1,\omega_0,\omega_0}=1$ and for $\omega>\omega_0$ it can be simply shown that
		\begin{equation*}
		a_{\omega,i} = \frac{f(\omega-1)}{f(\omega)-f(i)}a_{\omega-1,i}=\frac{\omega-1}{\omega-i}a_{\omega-1,i}\,, \quad i=\omega_0,\,\omega_0+1,\dots,\,\omega-1
		\end{equation*}
		Also  
		\begin{equation*}
		\sum_{i=\omega_0}^{\omega} {\omega-1 \choose \omega-\omega_0}{\omega -\omega_0 \choose i-\omega_0 }(-1)^{j-\omega_0} = {\omega-1 \choose \omega-\omega_0} \sum_{i=0}^{\omega-\omega_0} {\omega -\omega_0 \choose i}(-1)^{i}=(1-1)^{\omega-\omega_0} =0
		\end{equation*}
		and thus $a_{\omega,\omega} = - \sum_{i=\omega_0}^{\omega-1} a_{\omega,i} \,.$
	\end{Egs*}
	
	\section{Approximating the Probability Mass Function of the Feedback Process with Power Law Feedback}
	
		\label{sec_master_eqn}

We first present two further corollaries to Corollary \ref{gen_mast_sum}. 
	
	\begin{Cor}
		For the $a_{\omega,i}$ defined in Corollary \ref{gen_mast_sum} we have for $i=\omega_0, \, \omega_0+1, \, \dots ,\, \omega-1$
		\begin{equation*}
		a_{\omega,i} = 
		\frac{f(i)}{f(\omega)} 	\left(\prod_{j=i+1}^{\omega} \frac{1}{1-f(i)/f(j)}\right) a_{i,i} \,.
		\end{equation*}
		\label{cor_summands}
	\end{Cor}
	
	\begin{proof}
			\begin{align*}
			a_{\omega,i} &= \frac{f(\omega-1)}{f(\omega)-f(i)}
			a_{\omega-1,i} \\
			&= \frac{f(\omega-1)}{f(\omega)-f(i)}
			\frac{f(\omega-2)}{f(\omega-1)-f(i)} a_{\omega-2,i} \\
			&\vdots \\
			&= \frac{f(\omega-1) f(\omega-2) \dots f(i)}
			{(f(\omega)-f(i))(f(\omega-1)-f(i))
				\dots (f(i+1)-f(i))} a_{i,i} \\
			&= \left(\frac{f(i)}{f(\omega)}\right)
			\frac{f(i+1) \dots f(\omega-1)}
			{(f(i+1)-f(i)) \dots (f(\omega-1)-f(i))} \left(\frac{1}{1-(f(i)/f(\omega)} \right) a_{i,i}
			\\
			&= \frac{f(i)}{f(\omega)} 	\left(\prod_{j=i+1}^{\omega} \frac{1}{1-f(i)/f(j)}\right) a_{i,i} \,.
			\end{align*}
	\end{proof}
	
	We note that there does not seem to be any obvious general representation of the $a_{i,i}\,.$ With the $a_{\omega,i}$ defined in Corollary \ref{gen_mast_sum} and under the same mild condition on $f$ we have that
	\begin{Cor}
		For $W_0=\omega_0$ and $\omega>\omega_0$ the master equation \eqref{master_eqn} has solution
		\begin{equation}
		p_t(\omega) = 
		e^{-f(\omega_0) t} \left(\prod_{j=\omega_0+1}^{\omega} \frac{1}{1-f(\omega_0)/f(j)}\right)
		 f(\omega_0)/f(\omega) +\sum_{i=\omega_0+1}^{\omega} a_{\omega,i}e^{- f(i) t} \,.
		\label{master_approx_soln}
		\end{equation}
	\end{Cor}
	
	\begin{proof}
	For $\omega>\omega_0$ we can write \eqref{master_sum_soln} as
	\begin{equation*}
	p_t(\omega) = a_{\omega,\omega_0} e^{-f(\omega_0) t}
	+\sum_{i=\omega_0+1}^{\omega} a_{\omega,i}e^{- f(i) t} \,.
	\end{equation*}
	Now as $a_{\omega_0,\omega_0}=1$ from Corollary \ref{cor_summands} we have
	\begin{equation*}
	a_{\omega,\omega_0} = \frac{f(\omega_0)}{f(\omega)} 	\left(\prod_{j=\omega_0+1}^{\omega} \frac{1}{1-f(\omega_0)/f(j)}\right)  \,.
	\end{equation*}
	\end{proof}
	
	We find from simulations that for $\omega_0=1\,$, applied to the power law feedback function \eqref{pl_feedback_fn} with $\gamma>1\,$, and large enough $t\,,$ the first term of \eqref{master_approx_soln} (equivalently of \eqref{master_sum_soln})
	\begin{equation}
	\hat{p}_t(\omega) =	e^{-\eta \omega_0^{\gamma} t} \left(\prod_{j=\omega_0+1}^{\omega} \frac{1}{1-(\omega_0/j)^{\gamma}}\right)
	 (\omega_0/\omega)^{\gamma} \,.
	\label{master_pred_approx}
	\end{equation}
	is a good approximation to $p_t(\omega)$ \eqref{master_sum_soln},
	 see Figures \ref{urn_sum_pred_t} and \ref{urn_sum_preds}.  We note that \eqref{master_pred_approx} can be embedded into a regularly varying function, see the Appendix, 
	and so is dominated asymptotically\footnote{If we say a function $f(\omega)$ is asymptotically dominated by $\omega^{-\gamma}$ or asymptotically scales with $\omega^{-\gamma}$ we mean $f(\omega)$ is regularly varying with index $-\gamma\,,$ see the Appendix.} by $\omega^{-\gamma}\,.$ 
	For $\omega_0>1$ the estimate \eqref{master_pred_approx}  generally fails to be a good approximation but $p_t(\omega)$ still seems to asymptotically scale as $\omega^{-\gamma}$ for $\gamma>1\,$, see Figure \ref{urn_sum_pred_t1}. This indicates that the factor $1/f(\omega)$ seen in Corollary \ref{cor_summands}, at least for power law feedback \eqref{pl_feedback_fn}, plays a dominant role asymptotically in \eqref{master_sum_soln}.

	
	\begin{figure}[H]
		\centering
		\begin{minipage}{.47\textwidth}
	\centering
	\includegraphics[width=0.9\textwidth]{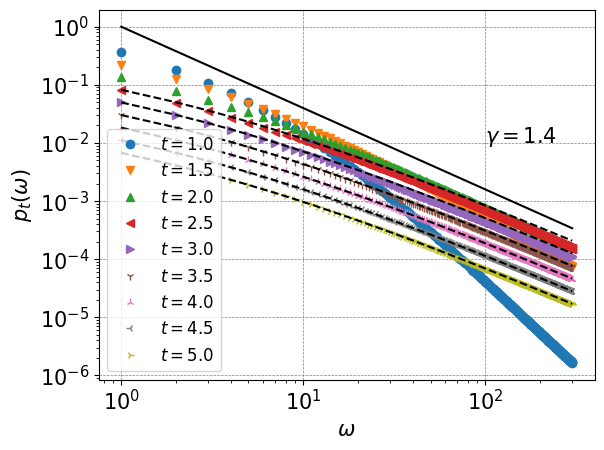}
	\caption{Numerical solution of the probability mass function sum formula solution \eqref{master_sum_soln} for $\omega_0=1\,$, for power law feedback \eqref{pl_feedback_fn} with $\eta=1$ and $\gamma=1.4$ for $t=1.0$ to $t=5.0$ in $0.5$ increments up to $\omega=300$ (for greater $\omega$ the simulation breaks down due to the magnitudes within the summands). Black dashed lines are the approximations \eqref{master_pred_approx} for $t\geq 2.5\,.$ Solid black line proportional to $\omega^{-\gamma}$ for comparison.} 
	\label{urn_sum_pred_t}
		\end{minipage}%
		\hfill
		\begin{minipage}{.47\textwidth}
	\centering
	\includegraphics[width=0.9\textwidth]{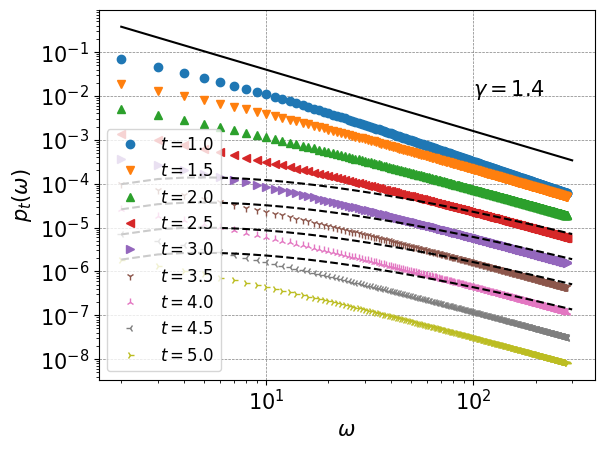}
	\caption{Numerical solution of the probability mass function sum formula solution \eqref{master_sum_soln} for $\omega_0=2\,$, for power law feedback \eqref{pl_feedback_fn} with $\eta=1$ and $\gamma=1.4$ for $t=1.0$ to $t=5.0$ in $0.5$ increments up to $\omega=300\,$. Black dashed lines are the approximations \eqref{master_pred_approx} for $t\geq 3.5\,$ which we see are not good. Solid black line proportional to $\omega^{-\gamma}$ for comparison. \newline  } 
	\label{urn_sum_pred_t1}
		\end{minipage}
	\end{figure}

	\begin{figure}[H]
	\centering
	\captionsetup{justification=centering}
	\includegraphics[width=0.45\textwidth]{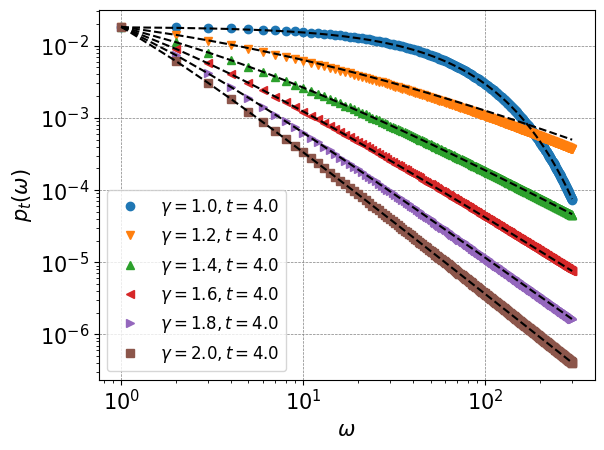}
	\caption{Numerical solution of the probability mass function sum formula solution \eqref{master_sum_soln} for $\omega_0=1\,$, for power law feedback \eqref{pl_feedback_fn} with $\eta=1$ and $\gamma$ going up in $0.2$ increments from $1$ to $2\,$ up to $\omega=300\,$. The black dashed line for $\gamma=1$ is an exponential fit approximating the geometric seen in Proposition \ref{Soln_partic_master}, while for $\gamma>1$ the black dashed lines are the approximations \eqref{master_pred_approx}.}
	\label{urn_sum_preds}
	\end{figure}
	
	\section{Approximating the Tail of the Feedback and Loser Process with Power Law Feedback}
	\label{sec_tail_sim}
	
	The time taken for an agent $W_t$ to visit all states $\omega_0, \, \omega_0 +1 , \dots $ is the cumulative sum of the holding times and described by the random variable
	\begin{equation*}
	T= \sum_{\omega=\omega_0}^{\infty} H_{\omega} \,.
	\end{equation*}
	If $T$ is finite this is called the \textbf{explosion time} and coincides with the monopoly agent case in Theorem \ref{feedback_cases}, Case \ref{feed_thm1},  (Theorem 1.1 in \cite{oliveira2008balls}). 
	We note that the expected explosion time is\footnote{Where linearity of expectation holds as a consequence of monotone convergence, see for example Section 5.6 of \cite{stirzaker1992probability}.}
		\begin{equation}
		\E[T] = \sum_{\omega=\omega_0}^{\infty} \E[H_{\omega}]
		=\sum_{\omega=\omega_0}^{\infty} \frac{1}{|g(\omega,\omega)|} =
		\sum_{\omega=\omega_0}^{\infty} \frac{1}{f(\omega)}\,.
		\label{exp_explosion}
		\end{equation}
		
	 Using the integral test, see e.g. Theorem 8.23 \cite{apostol_analysis1974}, we can bound the expected explosion time \eqref{exp_explosion} for the feedback process with power law feedback \eqref{pl_feedback_fn} and $\gamma>1$ as
		\begin{equation}
		t_{\gamma} :=\frac{1}{\eta \omega_0^{\gamma-1}(\gamma-1)} \leq 
		\sum_{\omega=\omega_0}^{\infty} \frac{1}{\eta \omega^{\gamma}} \leq 
		\frac{1}{\eta \omega_0^{\gamma}}+\frac{1}{\eta \omega_0^{\gamma-1}(\gamma-1)}
		\,
		\label{exp_explosion_pred}
		\end{equation}
		where $t_{\gamma}$ is defined as the lower bound of the expected explosion time. Notice from \eqref{exp_explosion_pred} that larger $\gamma$ gives a smaller expected explosion time.
		
	As discussed in Section 3.2 of \cite{oliveira2009onset} every agent will have a unique explosion time thus creating an ordering on the times an agent will explode. However as stated in Remark 2 in \cite{oliveira2009onset} one can only relate the feedback process to the feedback model up to the first explosion time.
	
	Now we would like to make inference on the losers $\textbf{L}$, see Definition \ref{def_losers}, in the feedback model by comparing to the feedback process.
	Let us define the loser process $\{\mathbf{L}_t \, : \, t=t_n\}$, where $t_n$ are the jump times, as a subset of the feedback process that have not exploded up to time $t$ i.e.
	\begin{equation*}
	\mathbf{L}_t = \{L_t\}=\{W_t \, : \, W_t < \infty\} \subseteq  \mathbf{W}_t\,.
	\end{equation*} 
	We now propose that $\mathbf{L}_t$ is a proxy for $\mathbf{L}$ within some time interval where an agent has exploded but not all agents have exploded. We note that for small enough $t$ no agents will have exploded so that $\mathbf{L}_t=\mathbf{W}_t$ and for very large $t$ all agents explode so that $ \mathbf{L}_t \rightarrow \emptyset$ (the empty set) as $t \rightarrow \infty\,.$
	
	We can simulate an agent $W_t$ of the feedback process and approximate an agent $L_t$ of the loser process as follows:
	\begin{enumerate}[label=\textbf{S.\arabic*}]
		\item Set $t_0=0$ and $W_0=\omega_0$ and set a jump time upper bound 
		$t_M<\infty$ and a  number of balls the agent collects upper bound $\omega_M< \infty\,.$ \label{sim1}
		\item Choose $H_{\omega_0} \sim \text{Exp}(f(\omega_0))$ and set $t_1=H_{\omega_0}$ and 
		$W_{t_1}=\omega_0+1\,.$ \label{sim2}
		\item Repeat \ref{sim2} so that in general $t_{n+1}=t_n+H_{\omega_n}$ and
		$W_{t_{n+1}}=W_{t_{n}}+1=\omega_0+(n+1)\,.$ \label{sim3}
		\item Stop at jump $n=k$ when either $t_{k}>t_M$ or $W_{t_{k}}>\omega_M\,.$ \label{sim4}
		\item A loser agent $L_t$  is approximated as an agent $W_t$ that has reached the time threshold before the ball threshold ($t_k>t_M$ but $W_{t_{k}} \leq \omega_M\,.$) In other words those agents $W_t$ that reach the ball threshold before the time threshold ($t_k \leq t_M$ but $W_{t_{k}}>\omega_M$) are assumed to be agents that have exploded by time $t_M$ and are removed. 
		\label{sim5}
	\end{enumerate}

	We can approximate the tail of the feedback process\footnote{Where $\sum_{k=\omega_0}^{\omega_0 -1}p_t(k) := 0$ and $\sum_{k=\omega_0}^{\omega_0 -1}\hat{p}_t(k) :=0\,.$}
	$\P(W_t \geq  \omega)= 1-\sum_{k=\omega_0}^{ \omega -1 } p_t(k)$
	using \eqref{master_pred_approx}:
	\begin{equation}
	\hat{\P}(W_t \geq  \omega) 
	= 1-\sum_{k=\omega_0}^{ \omega -1 } \hat{p}_t(k) \,.
	\label{tail_pred}
	\end{equation}
	From Corollary \ref{cor_mast_part} we can see for finite $\omega$ that
	$\lim_{t \rightarrow \infty} \P(W_t \geq \omega) = 1 \,.$
	 Figure \ref{wg2} shows the empirical tail of the aggregate of simulations \ref{sim1}-\ref{sim4} of $W_t$ for power law feedback \eqref{pl_feedback_fn} with $\gamma=2$ and Figure \ref{urn_tail_approx_g2} shows it's approximation \eqref{tail_pred}. We see a similar structure between the figures however the approximations are not perfect.
	 
	 Figure \ref{lg2} shows the empirical tail of the aggregate of the same simulations in Figure \ref{wg2} but with \ref{sim5} added at each simulation to remove the explosive agents and leave $L_t\,.$ Figure \ref{urn_agg_tails2} shows the aggregate of simulations \ref{sim1}-\ref{sim5} for different $\gamma$ values for power law feedback \eqref{pl_feedback_fn} and plots the empirical tail of $L_t\,.$ Both Figures \ref{lg2} and \ref{urn_agg_tails2} gives evidence that the losers scale as $\omega^{-(\gamma-1)}$ in support of Corollary \ref{cor_pl_two_ball} (Theorem 4 in \cite{oliveira2009onset} applied to the power law feedback function \eqref{pl_feedback_fn}) but for $N \geq 2$ agents.

		\begin{figure}[H]
			\centering
			\begin{minipage}{.47\textwidth}
			\centering
				\includegraphics[width=0.9\textwidth]{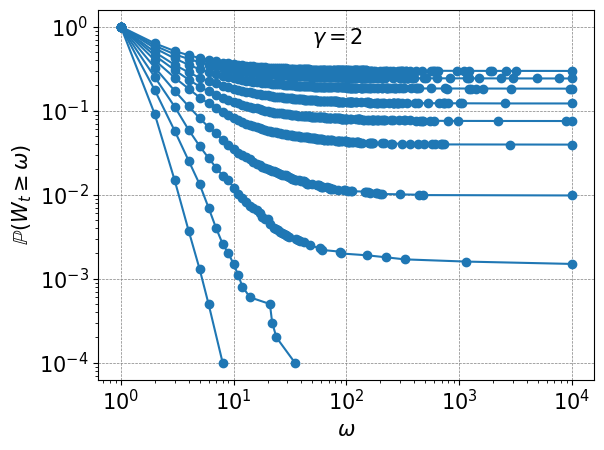}
				\caption{Empirical tail of aggregate of $10^4$ simulations of $W_t$, \ref{sim1}-\ref{sim4}, $\omega_0=1\,$, for power law feedback function \eqref{pl_feedback_fn} with $\eta=1$ and $\gamma=2\,$ until $t=t_M$ or $W_t=\omega_M=10^4\,$. The $t_M$ are chosen as $t_{\gamma}/10=0.1$ to $t_{\gamma}=1$ in increments of $0.1\,.$
				Higher $t$ pushes the tail monotonically upwards to $\P(W_t \geq \omega)=1\,$. }
				\label{wg2}
			\end{minipage}%
			\hfill
			\begin{minipage}{.47\textwidth}
			\centering
			\includegraphics[width=0.9\textwidth]{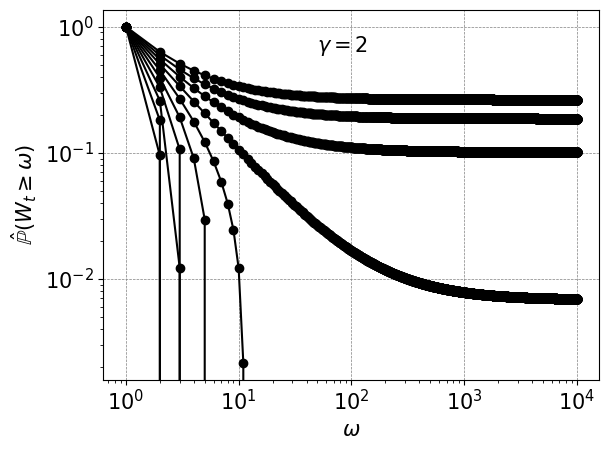}
			\caption{Plot of predicted tail \eqref{tail_pred}  for $W_t\,, \, \omega_0=1$ for power law feedback function \eqref{pl_feedback_fn} with $\eta=1$ and $\gamma=2\,$ and $t= t_{\gamma}/10=0.1$ to $t=t_{\gamma}=1$ in increments of $0.1\,.$ Higher $t$ pushes the predictions monotonically upwards to $\hat{\P}(W_t \geq \omega)=1\,$. We note there is error in the prediction (negative $\hat{\P}(W_t \geq \omega)$) for small $t\,.$
			 }
			\label{urn_tail_approx_g2}
			\end{minipage}
		\end{figure}
		
		\begin{figure}[H]
			\centering
			\begin{minipage}{.47\textwidth}
				\centering
				\includegraphics[width=0.9\textwidth]{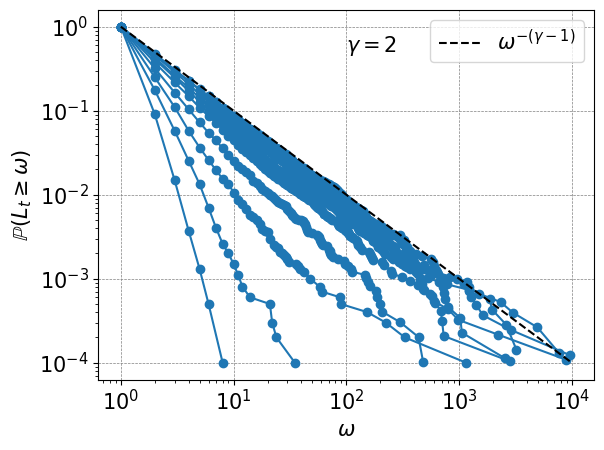}
				\caption{Plot of empirical tail of $L_t\,, \, \omega_0=1$ for power law feedback function \eqref{pl_feedback_fn} with $\eta=1$ and $\gamma=2\,.$ The $L_t$ are  $W_t$ used in Figure \ref{wg2} but adding step \ref{sim5} to exclude the $W_t$ that are likely to have exploded by $t_M$ i.e. those that reach $\omega_M=10^4$ before $t=t_M$ where $t_M=t_{\gamma}/10=0.1$ to $t_{\gamma}=1$ in increments of $0.1\,.$ Higher $t=t_M$ pushes \newline  $\hat{\P}(L_t \geq \omega)\,$ monotonically upwards to the dashed black line proportional to $w^{-(\gamma-1)}\,$. 
				 }
				\label{lg2}
			\end{minipage}
			\hfill
			\begin{minipage}{.47\textwidth}
				\centering
				\includegraphics[width=0.9\textwidth]{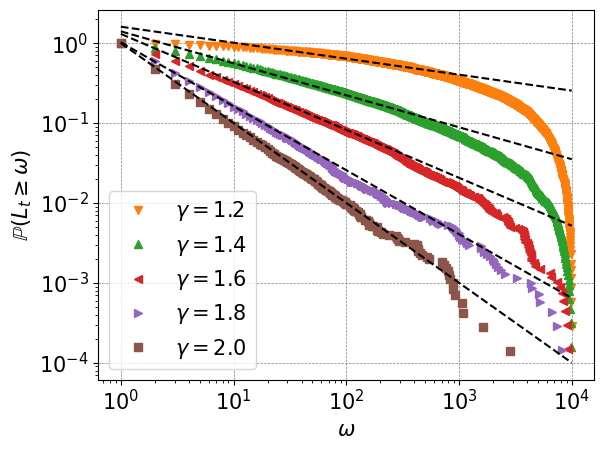}
				\caption{Empirical tail of $L_t$ from an aggregate of $10^4$ simulations of \ref{sim1}-\ref{sim5}, $\omega_0=1\,$, for power law feedback function \eqref{pl_feedback_fn} with $\eta=1$ and $\gamma$ going up in $0.2$ increments from $1.2$ to $2\,$ run until  $t_{M}=t_{\gamma}=1/(\gamma-1)$  but $W_t \leq \omega_M=10^4\,$. The black dashed lines are proportional to $w^{-(\gamma-1)}\,$.   \newline	\newline \newline}
				\label{urn_agg_tails2}
			\end{minipage}%
		\end{figure}
		
	\section{Conclusion}

	The form of the recursive mass function solution \eqref{master_sum_soln} to the master equation for the feedback process is not particularly tractable. However the first term \eqref{master_pred_approx} is analytical and for finite time and for the power law feedback function \eqref{pl_feedback_fn} with $\gamma>1$ is a regularly varying sequence asymptotically dominated by $\omega^{-\gamma} \propto 1/f(\omega) \,.$ 
	We showed numerically in Section \ref{sec_master_eqn} that this first term provides a good approximation under certain parametrisations. 
	We also provided simulations in Section \ref{sec_tail_sim} suggesting the tail of the losers scale as  $\omega^{-(\gamma-1)}$ for general $N \geq 2$ agents. This gives partial evidence that the corollary to Theorem 4 in Oliveira \cite{oliveira2009onset}, stated in Section \ref{sec_overview},  holds for power law feedback and $N \geq 2$ agents. As power law distributions appear  in many real world phenomena with possible feedback mechanisms we believe these are interesting results. 
	 
	 

	
	\renewcommand{\bibname}{References} 
	
	\bibliography{mybib}

\begin{thebibliography}{}

\bibitem[\protect\astroncite{Apostol}{1974}]{apostol_analysis1974}
Apostol, T.~M. (1974).
\newblock {\em Mathematical Analysis}.
\newblock Pearson, {2\textsuperscript{nd}} edition.

\bibitem[\protect\astroncite{Bingham et~al.}{1989}]{bingham1989regular}
Bingham, N.~H., Goldie, C.~M., Teugels, J.~L., and Teugels, J. (1989).
\newblock {\em Regular variation}.
\newblock Number~27. Cambridge university press.

\bibitem[\protect\astroncite{Clauset et~al.}{2009}]{clauset2009power}
Clauset, A., Shalizi, C.~R., and Newman, M.~E. (2009).
\newblock Power-law distributions in empirical data.
\newblock {\em SIAM review}, 51(4):661--703.

\bibitem[\protect\astroncite{Drinea et~al.}{2002}]{drinea2002balls}
Drinea, E., Frieze, A., and Mitzenmacher, M. (2002).
\newblock Balls and bins models with feedback.
\newblock In {\em SODA}, volume~2, pages 308--315.

\bibitem[\protect\astroncite{Feller}{1967}]{feller1967introduction}
Feller, W. (1967).
\newblock {\em An introduction to probability theory and its applications,
  Volume 1}.
\newblock John Wiley \& Sons, 3rd edition.

\bibitem[\protect\astroncite{Forbes}{2023}]{sam_git_code}
Forbes, S. (2023).
\newblock Git{H}ub simulations code.
\newblock \url{https://github.com/saf92/Urn-balls-in-bins-model}.

\bibitem[\protect\astroncite{Forbes and Grosskinsky}{2022}]{forbes2022study}
Forbes, S. and Grosskinsky, S. (2022).
\newblock A study of {UK} household wealth through empirical analysis and a
  non-linear {K}esten process.
\newblock {\em Plos one}, 17(8):e0272864.

\bibitem[\protect\astroncite{Galambos and Seneta}{1973}]{galambos1973regularly}
Galambos, J. and Seneta, E. (1973).
\newblock Regularly varying sequences.
\newblock {\em Proceedings of the American Mathematical Society},
  41(1):110--116.

\bibitem[\protect\astroncite{Gardiner et~al.}{1985}]{gardiner1985handbook}
Gardiner, C.~W. et~al. (1985).
\newblock {\em Handbook of stochastic methods}, volume~3.
\newblock springer Berlin.

\bibitem[\protect\astroncite{Gottfried and
  Grosskinsky}{2023}]{gottfried2023asymptotics}
Gottfried, T. and Grosskinsky, S. (2023).
\newblock Asymptotics of generalized {P\'{o}lya} urns with non-linear feedback.
\newblock {\em arXiv preprint arXiv:2303.01210}.

\bibitem[\protect\astroncite{Mori et~al.}{2020}]{mori2020common}
Mori, T., Smith, T.~E., and Hsu, W.-T. (2020).
\newblock Common power laws for cities and spatial fractal structures.
\newblock {\em Proceedings of the National Academy of Sciences},
  117(12):6469--6475.

\bibitem[\protect\astroncite{Oliveira}{2008}]{oliveira2008balls}
Oliveira, R. (2008).
\newblock Balls-in-bins processes with feedback and {B}rownian motion.
\newblock {\em Combinatorics, Probability and Computing}, 17(1):87--110.

\bibitem[\protect\astroncite{Oliveira}{2009}]{oliveira2009onset}
Oliveira, R.~I. (2009).
\newblock The onset of dominance in balls-in-bins processes with feedback.
\newblock {\em Random Structures \& Algorithms}, 34(4):454--477.

\bibitem[\protect\astroncite{Pemantle}{2007}]{pemantle2007survey}
Pemantle, R. (2007).
\newblock A survey of random processes with reinforcement.

\bibitem[\protect\astroncite{Severo}{1969}]{severo1969recursion}
Severo, N.~C. (1969).
\newblock A recursion theorem on solving differential-difference equations and
  applications to some stochastic processes.
\newblock {\em Journal of Applied Probability}, 6(3):673--681.

\bibitem[\protect\astroncite{Stirzaker and
  Grimmett}{1992}]{stirzaker1992probability}
Stirzaker, D. and Grimmett, G. (1992).
\newblock Probability and random processes.

\bibitem[\protect\astroncite{Vallejos et~al.}{2018}]{vallejos2018agent}
Vallejos, H.~A., Nutaro, J.~J., and Perumalla, K.~S. (2018).
\newblock An agent-based model of the observed distribution of wealth in the
  {United States}.
\newblock {\em Journal of Economic Interaction and Coordination}, 13:641--656.

\bibitem[\protect\astroncite{Vermeulen}{2018}]{vermeulen2018fat}
Vermeulen, P. (2018).
\newblock How fat is the top tail of the wealth distribution?
\newblock {\em Review of Income and Wealth}, 64(2):357--387.

\end{thebibliography}
	
	\bibliographystyle{apa}

\appendix
			
		\section*{Appendix: Regularly Varying Functions and Sequences}
		\label{app_reg_var}
		
		For a background to regular variation see  \cite{bingham1989regular}.
		\begin{Def}
			A positive, measureable function $f$ defined on $[k,\infty)$ for some $k\geq0$ is \textbf{regularly varying} if for all $\lambda>0$
			\begin{equation*}
			\lim_{x \rightarrow \infty} \frac{f(\lambda x)}{f(x)} = \lambda^{\rho}
			\end{equation*}
			for some index $\rho \in \R\,.$
			\label{def_reg_var}
		\end{Def}
		
		For a function $l$ satisfying the conditions in Definition \ref{def_reg_var} with the particular case of $\rho=0$ so that
			\begin{equation*}
			\lim_{x \rightarrow \infty} \frac{l(\lambda x)}{l(x)} = 1
			\end{equation*}
		then $l$ is called \textbf{slowly varying}. In particular $l(x):= f(x)/x^{\rho}$ where $f$ is some regularly varying function with index $\rho$ is a slowly varying function. Thus any regularly function $f$ can be written as
		\begin{equation*}
		f(x) = l(x)x^{\rho}
		\end{equation*}
		where $l$ is a slowly varying function. By the representation theorem we have for all $\gamma>0$ that
		$l(x)x^{-\gamma} \rightarrow 0$ as $x \rightarrow \infty$ thus if the index is $\rho=-\gamma<0$ then the regularly varying function $f(x)=l(x)x^{-\gamma}$ is dominated asymptotically by $x^{-\gamma}\,.$ 
		
		Now suppose we have a sequence $\{\theta(\omega) \, : \, \omega=\omega_0, \omega_0+1, \dots\}:=\{\theta(\omega)\}$ of positive terms. 
		
		\begin{Def}[See \cite{galambos1973regularly}]
			We call the sequence $\{\theta(\omega)\}$ a \textbf{regularly varying sequence} if there exists a sequence $\{\alpha(\omega)\}$ of positive terms such that
			\begin{enumerate}
				\item \begin{equation*}
				\theta(\omega) \simeq k \alpha(\omega) \quad \text{some } k>0
				\end{equation*}
				\item 
				\begin{equation*}
				\lim_{\omega \rightarrow \infty} 
				\omega(1-\alpha(\omega-1)/\alpha(\omega)) = \rho \quad \text{some index } \rho \in \R
				\end{equation*}
			\end{enumerate}
		\end{Def}
		
		If $\{\theta(\omega)\}$ is a regular varying sequence with index $\rho$ then it can be embedded into a regularly varying function with index $\rho\,$ \cite{galambos1973regularly}. In particular the function $f(x)=\theta(\lfloor x \rfloor)$ for $x \in [\omega_0,\infty)$ is a regularly varying function with index $\rho\,$ \cite{galambos1973regularly}.
		
		Let us show that the sequence $\{\hat{p}(\omega)\}$ defined by \eqref{master_pred_approx} is a regularly varying sequence. With 
		\begin{equation*}
		\alpha(\omega) = \left(\prod_{j=\omega_0+1}^{\omega} \frac{1}{1-(\omega_0/j)^{\gamma}}\right) \frac{1}{\omega^{\gamma}}
		\end{equation*}
		then $\{\alpha(\omega)\}$ is a positive sequence and we see
		\begin{equation*}
		\hat{p}(\omega) \simeq e^{-\eta \omega_0^{\gamma}t} \omega_0^{\gamma}\alpha(\omega)
		\end{equation*}
		satisfies the first condition of a regularly varying sequence. Now for $\gamma>1$ we have
		\begin{align*}
		\omega(1-\alpha(\omega-1)/\alpha(\omega)) &= 
		\omega\left(1-\frac{1/(\omega-1)^{\gamma}}{1/(1-(\omega_0/\omega)^{\gamma}) \cdot 1/\omega^{\gamma}} \right)\\
		&= \omega\left(1-\frac{\omega^{\gamma}-\omega_0^{\gamma}}{(\omega-1)^{\gamma}}
		\right) \\
		&=\omega\left(1-\frac{1}{(1-1/\omega)^{\gamma}} - \frac{\omega_0}{(\omega-1)^{\gamma}} \right) \\
		&=\omega\left(1-(1+\gamma/\omega+\mathcal{O}(1/\omega^{\gamma})) - \frac{\omega_0}{(\omega-1)^{\gamma}} \right) \\
		&= -\gamma +\mathcal{O}(1/\omega^{\gamma-1}) \\
		& \rightarrow -\gamma \quad \text{ as } \omega \rightarrow \infty \,.
		\end{align*}
		Thus the second condition is also satisfied and so $\{\hat{p}(\omega)\}$ is a regularly varying sequence and thus can be embedded into a regularly varying function with index $-\gamma\,.$


\end{document}